\documentclass[12pt]{article}
\usepackage{amssymb, amsmath, latexsym, a4}
\usepackage[latin1]{inputenc}
\usepackage{graphicx}
\usepackage[all]{xy}
\usepackage[usenames]{color}

\newenvironment{proof}[1][Proof]{\noindent\textbf{#1.} }{\ \rule{0.5em}{0.5em}}

\newtheorem{De}{Definition}[section]
\newtheorem{Th}[De]{Theorem}
\newtheorem{Pro}[De]{Proposition}
\newtheorem{Le}[De]{Lemma}
\newtheorem{Co}[De]{Corollary}
\newtheorem{Rem}[De]{Remark}
\newtheorem{Ex}[De]{Example}

\newcommand{\im}{\ensuremath{\mathsf{Im\,}}}

\renewcommand{\ker}{\ensuremath{\mathsf{Ker\,}}}

\newcommand{\Lie}{\ensuremath{\mathsf{Lie}}}

\newcommand{\Lieh}{\ensuremath{\mathfrak{h}}}
\newcommand{\Lieg}{\ensuremath{\mathfrak{g}}}
\newcommand{\Lieq}{\ensuremath{\mathfrak{q}}}
\newcommand{\Liea}{\ensuremath{\mathfrak{a}}}
\newcommand{\Lieb}{\ensuremath{\mathfrak{b}}}
\newcommand{\Liem}{\ensuremath{\mathfrak{m}}}
\newcommand{\Lien}{\ensuremath{\mathfrak{n}}}
\newcommand{\Lief}{\ensuremath{\mathfrak{f}}}
\newcommand{\Lies}{\ensuremath{\mathfrak{s}}}
\newcommand{\Lier}{\ensuremath{\mathfrak{r}}}
\newcommand{\Liet}{\ensuremath{\mathfrak{t}}}
\newcommand{\Liep}{\ensuremath{\mathfrak{p}}}
\newcommand{\Leib}{\ensuremath{\mathsf{Leib}}}

\newcommand{\ze}{\cal Z}

\newbox\pullbackbox
\setbox\pullbackbox=\hbox{\xy 0;<1mm,0mm>: \POS(4,0)\ar@{-} (0,0) \ar@{-} (4,4)
\endxy}

\newdir{>>}{{}*!/3.5pt/:(1,-.2)@^{>}*!/3.5pt/:(1,+.2)@_{>}*!/7pt/:(1,-.2)@^{>}*!/7pt/:(1,+.2)@_{>}}
\newdir{ >>}{{}*!/8pt/@{|}*!/3.5pt/:(1,-.2)@^{>}*!/3.5pt/:(1,+.2)@_{>}}
\newdir{ |>}{{}*!/-3.5pt/@{|}*!/-8pt/:(1,-.2)@^{>}*!/-8pt/:(1,+.2)@_{>}}
\newdir{ >}{{}*!/-8pt/@{>}}
\newdir{>}{{}*:(1,-.2)@^{>}*:(1,+.2)@_{>}}
\newdir{<}{{}*:(1,+.2)@^{<}*:(1,-.2)@_{<}}

\newcommand{\q}{\frak q}

\newcommand{\Ker}{\ensuremath{\mathrm{Ker}}}

   \newcommand{\eq}{\frak q} 
  \newcommand{\eh}{\frak h}

\begin{document}

\centerline{\bf  THE SCHUR $\Lie$-MULTIPLIER OF LEIBNIZ ALGEBRAS}

\bigskip
\bigskip
\centerline{\bf J. M. Casas and M. A. Insua}

\bigskip
\centerline{Dpto. Matemática Aplicada, Universidade de Vigo,  E. E. Forestal}
\centerline{Campus Universitario A Xunqueira, 36005 Pontevedra, Spain}
\centerline{ {E-mail addresses}: jmcasas@uvigo.es, avelino.insua@gmail.com}
\medskip

\date{}

\bigskip \bigskip \bigskip

{\bf Abstract:}   For a free presentation $0 \to \Lier \to \Lief \to \Lieg \to 0$ of a Leibniz algebra $\Lieg$, the Baer invariant ${\cal M}^{\Lie}(\Lieg) = \frac{\Lier \cap [\Lief, \Lief]_{\Lie}}{[\Lief, \Lier]_{\Lie}}$   is called the {\it Schur multiplier} of $\Lieg$ relative to the Liezation functor or Schur $\Lie$-multiplier. For a two-sided ideal $\Lien$ of a Leibniz algebra $\Lieg$, we construct a four-term exact sequence relating the Schur $\Lie$-multiplier of $\Lieg$ and $\Lieg/\Lien$, which is applied to study and  characterize $\Lie$-nilpotency, $\Lie$-stem covers and $\Lie$-capability of Leibniz algebras.
\bigskip

\bigskip

\bigskip


\section{Introduction}
In \cite{CVDL} the general theory of central extensions relative to a chosen subcategory of a base category introduced in \cite{JK} was considered in the context of semi-abelian categories \cite{JMT} relative to a Birkhoff subcategory. Examples like groups vs. abelian gropus, Lie algebras vs. vector spaces are absolute, meaning that they fit in the general theory when the considered Birkhoff subcategory is the subcategory of all abelian objects. An example of non absolute case is the category of Leibniz algebras together with the Birkhoff subcategory of Lie algebras. The general theory provides the notions of relative central extension and relative commutator wit respect to the Liezation functor $(-)_{\Lie} : \Leib \to \Lie$ which assigns to a Leibniz algebra $\Lieg$ the Lie algebra ${\Lieg_\Lie} = \Lieg / \Lieg^{\rm ann}$, where $ \Lieg^{\rm ann} = \langle \{ [x,x] : x \in \Lieg \} \rangle$.

Recently in  \cite{BC, CKh} properties concerning central extensions and commutators where translated from the absolute case to the relative one. In concrete the characterization of central extensions, capability and nilptency of Leibniz algebras relative to the Liezation functor  by means homological machinery was provided as well as a systematic study of isoclinism of Leibniz algebras relative to the Liezation functor. All these relative notions with respect the Liezation functor are named with the prefix $\Lie$-.

Our goal in the current paper is to continue analyzing the behavior of absolute properties when they fit into the relative context. In concrete we study the application of the relative Schur multiplier with respect to the Liezation functor  of a Leibniz algebra, called Schur $\Lie$-multiplier,
to characterize $\Lie$-nilpotency, $\Lie$-stem covers and $\Lie$-capability of Leibniz algebras.

In order to reach our goals, the content is organized as follows: in section \ref{preliminaries} we recall basic facts concerning the Liezation functor like the notions of $\Lie$-central extension, $\Lie$-commutator and $\Lie$-homology of Leibniz algebras. In subsection \ref{nil} we recall the notion of $\Lie$-nilpotent Leibniz algebra, providing the classification up to dimension 4 of complex $\Lie$-nilpotent non-Lie Leibniz algebras, as well as the characterization of $\Lie$-nilpotency through the $\Lie$-normalizer condition. In section \ref{Schur multiplier}, for a free presentation $0 \to \Lier \to \Lief \to \Lieg \to 0$ of a Leibniz algebra $\Lieg$, the Baer invariant ${\cal M}^{\Lie}(\Lieg) = \frac{\Lier \cap [\Lief, \Lief]_{\Lie}}{[\Lief, \Lier]_{\Lie}}$ \cite{CVDL}   called the {\it Schur \Lie-multiplier} of $\Lieg$, and for a two-sided ideal $\Lien$ of a Leibniz algebra $\Lieg$, we construct a four-term exact sequence relating the Schur $\Lie$-multiplier of $\Lieg$ and $\Lieg/\Lien$ (see (\ref{exact seq})), which is useful to characterize $\Lie$-nilpotent Leibniz algebras. By the way, in case of finite dimension, some formulas concerning the dimension of the Schur $\Lie$-multiplier are derived. In section \ref{stem cover} we deal with the interplay between the Schur $\Lie$-multiplier and $\Lie$-stem covers. In particular, we prove that any finite-dimensional Leibniz algebras has at least one $\Lie$-stem cover. Section \ref{capability} is devoted to analyze the conections between the precise $\Lie$-center introduced in \cite{CKh} and the Schur $\Lie$-multiplier, from here we obtain a characterization of $\Lie$-capable Leibniz algebras.


\section{Preliminary results on Leibniz algebras} \label{preliminaries}
We fix $\mathbb{K}$ as a ground field such that $\frac{1}{2} \in \mathbb{K}$. All vector spaces and tensor products are considered over $\mathbb{K}$.

A \emph{Leibniz algebra} \cite{Lo 1, Lo 2, LP} is a vector space $\eq$  equipped with {a linear map} $[-,-] : \Lieq \otimes \Lieq \to \Lieq$, usually called the \emph{Leibniz bracket} of $\eq$,  satisfying the \emph{Leibniz identity}:
\[
 [x,[y,z]]= [[x,y],z]-[[x,z],y], \quad x, y, z \in \Lieq.
\]

Leibniz algebras { constitute a variety of $\Omega$-groups \cite{Hig}, hence it is a semi-abelian variety \cite{CVDL, JMT}} denoted by {\Leib}, whose morphisms are linear maps that preserve the Leibniz bracket.

 A subalgebra ${\eh}$ of a Leibniz algebra ${\Lieq}$ is said to be \emph{left (resp. right) ideal} of ${\Lieq}$ if $ [h,q]\in {\eh}$  (resp.  $ [q,h]\in {\eh}$), for all $h \in {\eh}$, $q \in {\Lieq}$. If ${\eh}$ is both
left and right ideal, then ${\eh}$ is called \emph{two-sided ideal} of ${\Lieq}$. In this case $\Lieq/\Lieh$ naturally inherits a Leibniz algebra structure.
{Note that the notion of two-sided ideal coincides with the categorical notion of normal subobject (i.e. kernel) so that the quotient $\Lieq/\Lieh$ is the cokernel of the kernel $\Lieh \to \Lieq$.}

For a Leibniz algebra ${\Lieq}$, we denote by ${\Lieq}^{\rm ann}$ the subspace of ${\Lieq}$ spanned by all elements of the form $[x,x]$, $x \in \Lieq$. Further, we consider
\[
Z^r({\Lieq}) = \{a \in {\Lieq} \mid [x,a] = 0,\ x \in {\Lieq}\},\quad Z({\Lieq}) = \{a \in {\Lieq} \mid [x,a] = 0=[a,x],\ x \in {\Lieq}\}
\]
 and call the \emph{right center} and  \emph{center}  of $\Lieq$, respectively.  It is proved in \cite[Lemma 1.1]{KP} that both ${\Lieq}^{\rm ann}$ and $Z^r({\Lieq})$ are two-sided ideals of $\Lieq$. It is obvious that $Z({\Lieq})$ is also a  two-sided ideal of $\Lieq$.

\subsection{The Liezation functor}
Given a Leibniz algebra $\Lieq$, it is clear that the quotient ${\Lieq}_ {_{\rm Lie}}=\Lieq/{\Lieq}^{\rm ann}$ is a Lie algebra. This defines the so-called  \emph{Liezation functor} $(-)_{\Lie} : {\Leib} \to {\Lie}$, which assigns to a Leibniz algebra $\Lieq$ the Lie algebra ${\Lieq}_{_{\rm Lie}}$. Moreover, the canonical epimorphism  ${\Lieq} \twoheadrightarrow {\Lieq}_ {_{\rm Lie}}$ is universal among all homomorphisms from $\Lieq$ to a Lie algebra, implying that the Liezation functor is left adjoint to the inclusion functor $ {\Lie} \hookrightarrow {\Leib}$. These facts provide the following adjoint pair
\begin{equation} \label{adjoint pair}
\xymatrix@1{{\Leib} \ar@<1ex>[r]^-{(-)_{\Lie}} \ar@{}[r]|-{\perp} & {\Lie} \ar@<1ex>[l]^-{\supset}},
\end{equation}
 The general theory of central extensions relative to a chosen subcategory of a base category introduced in \cite{JK}, was adapted to the setting of semi-abelian categories relative a Birkhoff subcategory in \cite{CVDL}. Since $\Lie$ is a subvariety of $\Leib$, then it is a Birkhoff subcategory of $\Leib$, then the particular case corresponding to the adjoint pair (\ref{adjoint pair})  provides the following concepts relative to the Liezation functor $(-)_{\Lie}$ (see \cite{CKh, CVDL} for details), hence named with the prefix \Lie.

 For a Leibniz algebra {\Lieq} and two-sided ideals  ${\Liem}$ and ${\Lien}$ of  ${\Lieq}$, the \Lie-{\it centralizer} of ${\Liem}$ and ${\Lien}$ over  ${\Lieq}$ is
\[
C_{\Lieq}^{\Lie}({\Liem} , {\Lien}) = \{q \in {\Lieq} \mid  \; [q, m] + [m,q] \in {\Lien}, \; \text{for all} \;
m \in {\Liem} \} \; .
\]

\noindent The \Lie-{\it commutator} $[\Liem,\Lien]_{\Lie}$ is the subspace  of $\Lieq$ spanned by all elements of the form $[m,n]+[n,m]$, $m \in \Liem$, $n \in \Lien$.

\begin{Le}\label{Lemma3.3} \cite[Lemma 1]{CKh}  Let $\Lieq$ be a Leibniz algebra and $\Liem$, $\Lien$ be two-sided ideals of $\Lieq$. Then both $C_{\Lieq}^{\Lie}({\Liem} , {\Lien})$  and $[\Liem,\Lien]_{\Lie}$ are  two-sided ideals of $\Lieq$. Moreover, $ Z(\Lieq)\subseteq C_{\Lieq}^{\Lie}({\Liem} , {\Lien}) $ and $[\Liem,\Lien]_{\Lie}\subseteq Z^r(\Lieq)$.
\end{Le}

In particular, the two-sided ideal $C_{\Lieq}^{\Lie}(\Lieq , 0)$ is the \emph{$\Lie$-center} of the Leibniz algebra $\Lieq$ and it will be denoted by $Z_{\Lie}(\Lieq)$, that is,
\[
Z_{\Lie}(\Lieq) =  \{ z\in \Lieq\,|\,\text{$[q,z]+[z,q]=0$ for all $q\in \Lieq$}\}.
\]

An extension of Leibniz algebras $f : \Lieg \twoheadrightarrow \Lieq$ with ${\Lien} =  \Ker(f)$ is said to be  $\Lie$-{\it central} if
 $\Lien \subseteq Z_{\Lie}(\Lieg)$, equivalently  $[\Lien,\Lieg]_{\Lie} = 0$ (see \cite[Proposition 1]{CKh}).

Homological machinery relative to the Liezation functor as a particular case of the semi-abelian framework \cite{Ev, EVDL, EVDL1} provide that the first and second  homologies relative to the Liezation functor of a Leibniz algebra $\Lieg$ are given by  $H_1(\Lieg, (-)_{\Lie}) \equiv HL^{\Lie}_1(\Lieg) =\Lieg_{\Lie}$ and $H_2(\Lieg, (-)_{\Lie}) \equiv HL^{\Lie}_2(\Lieg) \cong \frac{\Lier \cap [\Lief, \Lief]_{\Lie}}{[\Lief, \Lier]_{\Lie}}$, for any free presentation $0 \to \Lier \to \Lief \to \Lieg \to 0$. Moreover, these relative invariants are related by the six-term exact sequence
\begin{equation}\label{six-term}
 \Lien\otimes \Lieg_{\Lie} \longrightarrow HL^{\Lie}_2(\Lieg) \longrightarrow HL^{\Lie}_2(\q) \stackrel{\theta}\longrightarrow \frak{n} \longrightarrow HL^{\Lie}_1(\Lieg) \longrightarrow HL^{\Lie}_1(\q) \longrightarrow 0.
\end{equation}
provided that $f:\Lieg\twoheadrightarrow \Lieq$,  with $\Lien=\Ker(f)$, is a  $\Lie$-central extension of Leibniz algebras (see \cite[Proposition 2]{CKh}).

\subsection{$\Lie$-nilpotent Leibniz algebras} \label{nil}

The notion of relative commutator allow the introduction of lower and upper \Lie-central series and, consequently, the notion of \Lie-nilpotent Leibniz algebra (see \cite{CKh} for details).

\begin{De}\label{lcs}
The lower $\Lie$-central series of a Leibniz algebra $\Lieq$  is the sequence
\[
\cdots \trianglelefteq{\Lieq}^{[i]} \trianglelefteq \cdots \trianglelefteq {\Lieq}^{[2]}  \trianglelefteq {\Lieq}^{[1]}
\]
of two-sided ideals of $\Lieq$ defined inductively by
 \[
 {\Lieq}^{[1]} = {\Lieq} \quad \text{and} \quad
{\Lieq}^{[i]} =[{\Lieq}^{[i-1]},{\Lieq}]_{\Lie} ,  \   i \geq 2  .
 \]

A Leibniz algebra ${\Lieq}$ is said to be $\Lie$-nilpotent with class of nilpotency $k$ if and only if ${\Lieq}^{[k+1]}=0$ and
${\Lieq}^{[k]}\ne 0$.
\end{De}

\begin{De}
The upper $\Lie$-central series of a Leibniz algebra ${\Lieq}$ is the sequence of
 two-sided ideals
 \[
{\ze}_0^{\Lie}({\Lieq}) \trianglelefteq {\ze}_1^{\Lie}({\Lieq}) \trianglelefteq \cdots \trianglelefteq {\ze}_i^{\Lie}({\Lieq}) \trianglelefteq \cdots
\]
 defined inductively by
\[
{\ze}_0^{\Lie}({\Lieq}) = 0 \quad \text{and} \quad
 {\ze}_{i}^{\Lie}({\Lieq}) = C_{\Lieq}^{\Lie}({\Lieq},{\ze}_{i-1}^{\Lie}({\Lieq})) , \  i \geq 1 .
 \]

 A Leibniz algebra ${\Lieq}$ is said to be $\Lie$-nilpotent with class of $\Lie$-nilpotency k if and only if
${\ze}_k^{\Lie}({\Lieq}) = {\Lieq}$ and ${\ze}_{k-1}^{\Lie}({\Lieq}) \ne {\Lieq}$.
\end{De}

\begin{Pro} \cite[Proposition 10]{CKh}\label{nilpotency}\
\begin{enumerate}
\item[(a)] If ${\Lieq}/Z_{\Lie}({\Lieq}$) is a $\Lie$-nilpotent Leibniz algebra, then ${\Lieq}$ is a $\Lie$-nilpotent Leibniz algebra.

\item[(b)] If ${\Lieq}$ is a $\Lie$-nilpotent and non trivial Leibniz algebra, then $Z_{\Lie}({\Lieq})\not = 0$.

\item[(c)] If $\Lieg\twoheadrightarrow\Lieq$ is a $\Lie$-central extension of a $\Lie$-nilpotent Leibniz algebra $\Lieq$, then $\Lieg$ is $\Lie$-nilpotent as well.

\end{enumerate}
\end{Pro}

\begin{Ex}\
\begin{enumerate}
\item[(a)] {Lie algebras  are $\Lie$-nilpotent Leibniz algebras of class 1. In particular, vector spaces considered as abelian Lie (Leibniz) algebras are $\Lie$-nilpotent Leibniz algebras of class 1.}
    \item[(b)]  Subalgebras and images by homomorphisms of $\Lie$-nilpotent Leibniz algebras are $\Lie$-nilpotent Leibniz algebras.
\item[(c)] Intersection and sum of $\Lie$-nilpotent two-sided ideals of a Leibniz algebra are $\Lie$-nilpotent two-sided ideals as well.

\item[(d)] From the classifications of two-dimensional Leibniz algebras in \cite{Cu}, three-dimensional Leibniz algebras in \cite{CILL}, four dimensional Leibniz algebras in \cite{CK, DMS}  and having in mind that all Lie algebras are \Lie-nilpotent Leibniz algebras, in the following table we present the isomorphism classes of low-dimensional \Lie-nilpotent  non-Lie Leibniz algebras over the field $\mathbb{C}$ of complex numbers.
\[
\begin{tabular}{|l|l|c|}
  \hline
{\bf basis} & {\bf multiplication} & {\bf class of} \\
 & & {\bf \Lie-nilp.}\\
  \hline  \hline
   $\{a_1,a_2\}$ & $[a_2,a_2]=a_1$ & 2 \\
   \hline  \hline
   $\{a_1,a_2,a_3\}$ &  $[a_2,a_2]= \gamma a_1; [a_3,a_2]=a_1; [a_3,a_3]=a_1$& 2 \\
   \hline
     & $[a_3,a_3]=a_1$ & 2\\
    \hline
     & $[a_2,a_2]=a_1; [a_3,a_3]=a_1$ & 2\\
    \hline
    & $[a_2,a_3]=a_2; [a_3,a_2]=-a_2; [a_3,a_3]=a_1$& 2\\
     \hline
     & $[a_1,a_3]=a_2;  [a_3,a_3]=a_1$ & 3\\
   \hline  \hline

   $\{a_1,a_2,a_3,a_4\}$ & $[a_1,a_3]=a_4; [a_3,a_2]=a_4$ &2 \\
   \hline
    & $[a_1,a_3]=a_4; [a_2,a_2]=a_4; [a_2,a_3]=a_4;$ & 2\\
    & $ [a_3,a_1]= a_4; [a_3,a_2]= -a_4$ & \\
    \hline
  &$[a_1,a_2]=a_4;[a_2,a_1]=-a_4; [a_3,a_3]=a_4$ & 2 \\
  \hline
   & $[a_1,a_2]=a_4; [a_2,a_1]=-a_4;[a_2,a_2]= a_4; [a_3,a_3]=a_4$ & 2 \\
  \hline
   & $[a_1,a_2]=a_4; [a_2,a_1]= c a_4; [a_3,a_3]=a_4, c \in \mathbb{C} \backslash\{1,-1\}$ & 2\\
  \hline
   & $[a_1,a_1]=a_4; [a_2,a_2]=a_4; [a_3,a_3]=a_4$& 2\\
  \hline
   & $[a_1,a_1]=a_2; [a_1,a_2]=a_3; [a_1,a_3]=a_4$& 4\\
  \hline
   &  $[a_1,a_1]=a_4; [a_1,a_2]=a_3; [a_2,a_1]= - a_3$ & 2\\
  \hline
 & $[a_1,a_1]=a_4; [a_1,a_2]=a_3; [a_2,a_1]= - a_3; [a_2,a_2]=a_4$ & 2\\
  \hline
   &  $[a_1,a_1]=a_4; [a_1,a_2]=a_3; [a_2,a_1]= - a_3;$  & 2\\
   & $[a_1,a_3]=a_4; [a_3,a_1]=-a_4$ &\\
  \hline

   \hline
   \end{tabular}
   \]

   \[
\begin{tabular}{|l|l|c|}
  \hline
 {\bf basis} & {\bf multiplication} & {\bf class of} \\
 & & {\bf \Lie-nilp.}\\
  \hline  \hline

    &  $[a_1,a_2]=a_3; [a_2,a_1]= - a_3; [a_2,a_2]= a_4;$  & 2\\
   & $[a_1,a_3]=a_4; [a_3,a_1]=-a_4$ &\\
  \hline
    &  $[a_1,a_1]=a_4; [a_1,a_2]= a_3; [a_2,a_1]= -a_3 + a_4;$  & 2\\
   & $[a_1,a_3]=a_4; [a_3,a_1]=-a_4$ &\\
  \hline
   &  $[a_2,a_2]=a_4; [a_1,a_2]= a_3; [a_2,a_1]= -a_3 + a_4;$  & 2\\
   & $[a_1,a_3]=a_4; [a_3,a_1]=-a_4$ &\\
   \hline
   & $[a_1,a_1]=a_3;  [a_1,a_2]=  a_4$ & 2\\
    \hline
   & $[a_1,a_1]=a_3;  [a_2,a_1]=  a_4$ & 2\\
  \hline
   &  $[a_1,a_2]=a_4;  [a_2,a_1]=  a_3; [a_2,a_2] = -a_3$ & 2\\
  \hline
   & $[a_1, a_1]=a_3; [a_1,a_2]=a_4;  [a_2,a_1]= \alpha a_4, \alpha \in \mathbb{C} \backslash \{-1, 0\}$ & 2 \\
  \hline
   & $[a_1, a_1]=a_3; [a_1,a_2]= \alpha a_3;  [a_2,a_1]= a_4; $  & 2\\
   &  $[a_2, a_2] = - a_4, \alpha \in \mathbb{C} \backslash \{-1 \}$  & \\
  \hline
  & $[a_1, a_1]=a_3; [a_1,a_2]= a_3;  [a_2,a_1]= a_3 + a_4; [a_2, a_2] = a_4$  & 2\\
  \hline
   & $[a_1, a_2]=a_3; [a_1,a_3]= a_4$& 3\\
  \hline
   & $[a_1, a_2]=a_3; [a_1,a_3]= a_4; [a_2, a_2]=a_4$ & 3 \\
  \hline
   & $[a_1, a_2]=a_3; [a_1,a_3]= a_4; [a_2, a_1]=a_4$ & 3 \\
  \hline
   & $[a_1, a_2]=a_3; [a_2, a_1]=a_4; [a_2,a_2]=a_4;[a_1,a_3]= a_4$ & 3 \\
  \hline
  & $[a_1, a_1]=a_3; [a_2, a_1]=a_4; [a_1,a_3]= a_4$ & 3 \\
  \hline
   & $[a_1, a_1]=a_3; [a_2,a_2]=a_4;[a_1,a_3]= a_4$ & 3 \\
  \hline
    & $[a_1, a_2]=a_3; [a_2,a_1]=- a_3;[a_1,a_4]= a_1; [a_2,a_4]= - a_2;$ & 2\\
   & $[a_4, a_1]= - a_1; [a_4,a_2]= a_2;[a_4,a_4]= a_3$& \\
  \hline
   & $[a_1, a_1]= a_3; [a_2,a_4]= a_2;[a_4,a_2]= - a_2$ & 2\\
  \hline
   &  $[a_1, a_1]= a_3; [a_2,a_4]= a_2;[a_4,a_1]= a_3; [a_4,a_2]= - a_2;$ & 2\\
   & $[a_4, a_4] = \lambda a_3, \, \lambda \neq 0$ & \\
  \hline
   &  $[a_1, a_1]= a_3; [a_2,a_4]= a_2;[a_4,a_2]= - a_2; [a_4,a_4]= - 2a_3$  & 2 \\
  \hline
   & $[a_1, a_4]= a_1; [a_2,a_4]= \mu_2 a_2;[a_4,a_1]= - a_1; [a_4,a_2]= - \mu_2 a_2;$   & 2\\
   & $[a_4,a_4]= a_3$ & \\
  \hline
   &  $[a_1, a_4]= a_1; [a_4,a_1]= - a_1;[a_4,a_2]= a_3$  &2 \\
  \hline
   & $[a_1, a_4]= a_1 + a_2; [a_2,a_4]= a_2;[a_4,a_1]= - a_1 - a_2; $  & 2\\
   & $[a_4,a_2]= - a_2;[a_4, a_4] = a_3$ & \\
  \hline
   & $[a_1, a_4]= a_2; [a_3,a_4]= a_3;[a_4,a_1]= \alpha a_2; [a_4,a_3] = -a_3, \, \alpha \neq -1$   &2 \\
  \hline
   & $[a_1, a_4]= a_2; [a_3,a_4]= a_3;[a_4,a_1]= - a_2;$   & 2\\
   & $[a_4, a_3]= - a_3; [a_4,a_4]= a_2$  & \\
  \hline
   & $[a_1, a_4]=a_2; [a_3,a_4]= a_3;[a_4,a_3]= - a_3; [a_4,a_4]=a_1$   & 3 \\
  \hline
\end{tabular}
\]

 \end{enumerate}
\end{Ex}
 \bigskip

Now we complete the characterizations of $\Lie$-nilpotent Leibniz algebras given in \cite{CKh}.

\begin{Pro}\
\begin{enumerate}
\item [(a)] Let $\Lieh$ be a two-sided ideal of a Leibniz algebra $\Lieg$ such that $\Lieh \subseteq Z_{\Lie}(\Lieg)$. Then $\Lieg$ is \Lie-nilpotent if and only if $\Lieg/\Lieh$ is \Lie-nilpotent.

 \item[(b)]  Let $f : \Lieg \twoheadrightarrow \Lieq$ be a \Lie-central extension of a Leibniz algebra $\Lieq$. Then $\Lieg$ is \Lie-nilpotent if and only if $\Lieq$ is \Lie-nilpotent.
\end{enumerate}
\end{Pro}
\begin{proof}
{\it (a)} The quotient of \Lie-nilpotent  Leibniz algebras is \Lie-nilpotent as well. Conversely, there exist $k \in \mathbb{N}$  such that $(\Lieg/\Lieh)^{[k]} = 0$, hence $\Lieg^{[k]} \subseteq \Lieh \subseteq Z_{\Lie}(\Lieg)$. Then $\Lieg^{[k+1]} =[\Lieg^{[k]}, \Lieg]_{\Lie} \subseteq [\Lieh, \Lieg]_{\Lie}=0$.

{\it (b)} is a direct consequence of {\it (a)}.
\end{proof}

\begin{De}
Let $\Liem$ be a subset of a Leibniz algebra $\Lieq$. The \Lie-normalizer of $\Liem$ is the subset of $\Lieq$:
$$N_{\Lieq}^{\Lie}(\Liem) = \{ q \in \Lieq \mid [q,m]+[m, q]\in \Liem, \text{\rm for all}\ m \in \Liem \}$$
\end{De}

\begin{Rem}
When $\Liem$ is a subalgebra of $\Lieq$, then $N_{\Lieq}^{\Lie}(\Liem)$ is not necessarily a subalgebra of $\Lieq$ as the following example shows: let $\Lieq$ be the five-dimensional complex Leibniz algebra with basis $\{e_1, e_2, e_3,e_4,e_5 \}$ and bracket operation given by (see \cite{Om})
\[
 \begin{array}{lll} [e_2,e_1]=-e_3, &
 [e_1,e_2]=e_3,& [e_1,e_3]=-2e_1 \\

[e_3,e_1]=2e_1, & [e_3,e_2]=-2e_2, & [e_2,e_3]=2e_2 \\

[e_5,e_1]=e_4, & [e_4,e_2]=e_5,& [e_4,e_3]=-e_4\\ & &
 [e_5,e_3]=e_5 \end{array}
\]
Consider the subalgebra $\Liem = \langle \{ e_1 \} \rangle$, then $N_{\Lieq}^{\Lie}(\Liem) = \langle \{ e_1, e_2, e_3, e_4 \} \rangle$ which is not a subalgebra.

On the other hand, if $\Liem$ is a two-sided ideal of $\Lieq$, then $N_{\Lieq}^{\Lie}(\Liem)$ is a two-sided ideal of $\Lieq$, since it  coincides with $C_{\Lieq}^{\Lie}(\Liem,\Liem)$ and \cite[Lemma 1]{CKh}. Furthermore, if $\Liem$ is a subalgebra of $\Lieq$, then $\Liem \subseteq N_{\Lieq}^{\Lie}(\Liem)$.
\end{Rem}

\begin{De}
It is said that a Leibniz algebra $\Lieq$ satisfies the \Lie-normalizer condition if every proper subalgebra of $\Lieq$ is properly contained in its normalizer.
\end{De}

\begin{Pro}
If $\Lieq$ is a \Lie-nilpotent Leibniz algebra, then $\Lieq$ satisfies the \Lie-normalizer condition.
\end{Pro}
\begin{proof}
Let $\Lies$ be a proper subalgebra of $\Lieq$. Let $j \geq 1$ be the minimal integer such that $Z_j^{\Lie}(\Lieq)\nsubseteqq \Lies$ (there always exists such a $j$ thanks to \cite[Thoerem 4]{CKh}). Then $[\Lies, Z_j^{\Lie}(\Lieq)]_{\Lie} \subseteq [\Lieq, Z_j^{\Lie}(\Lieq)]_{\Lie} \subseteq Z_{j-1}^{\Lie}(\Lieq) \subseteq \Lies$. Thus $\Lies \subseteq \Lies + Z_j^{\Lie}(\Lieq) \subseteq N_{\Lieq}^{\Lie}(\Lies)$.
\end{proof}

\section{The Schur \Lie-multiplier of Leibniz algebras} \label{Schur multiplier}

For a free presentation $0 \to \Lier \to \Lief \stackrel{\rho}\to \Lieg \to 0$ of a Leibniz algebra $\Lieg$ and in analogy with the absolute case, the term {\large $\frac{\Lier \cap [\Lief, \Lief]_{\Lie}}{[\Lief, \Lier]_{\Lie}}$}
 is called the {\it Schur \Lie-multiplier} or the {\it Schur multiplier relative to the Liezation functor} of $\Lieg$, which is denoted by ${\cal M}^{\Lie}(\Lieg)$. As is reported in \cite{CKh}, the Schur \Lie-multiplier is isomorphic to $HL^{\Lie}_2(\Lieg)$ and it is a Baer invariant, which means that it does not depend on the chosen free presentation as explained for instance in \cite{EVDL}.

Our aim in this section is  to show the interplay between the Schur \Lie-multiplier and \Lie-nilpotent Leibniz algebras, as well as the obtention of several formulas concerning dimensions.

 \begin{Th} \label{Th 1}
 Let $\Lieg$ be a Leibniz algebra with a two-sided ideal $\Lieb$ and set the short exact sequence $0 \to \Lieb \to \Lieg \to \Liea \to 0$.  Then there exists a Leibniz algebra $\Lieq$ with a two-sided ideal $\Liem$ such that:
 \begin{enumerate}
 \item[(a)] $[\Lieg, \Lieg]_{\Lie} \cap \Lieb \cong \frac{\Lieq}{\Liem}$.
 \item[(b)] $\Liem \cong {\cal M}^{\Lie}(\Lieg)$.
 \item[(c)] ${\cal M}^{\Lie}(\Liea)$ is an epimorphic image of $\Lieq$.
 \end{enumerate}
 \end{Th}
\begin{proof}
Let $0 \to \Lier \to \Lief \stackrel{\rho}\to \Lieg \to 0$ be a free presentation of $\Lieg$ and consider the following diagram of free presentations:
\begin{equation} \label{free present diagr}
\xymatrix{& & 0 \ar[d] & 0 \ar[ld]\\
&  & {\Lier} \ar[d]\ar[ld] \\
0 \ar[r] & {\Lies} \ \ar[r] \ar[d] & {\Lief} \ar[d]^{\rho} \ar[rd]^{\pi \circ \rho} \\
0 \ar[r]& {\Lieb} \ar[r] \ar[d]& {\Lieg} \ar[r]^{\pi} \ar[d]& {\Liea} \ar[r] \ar[dr] & 0 \\
 & 0 & 0 &  & 0
}
\end{equation}
Then $\Liea \cong  \frac{\Lieg}{\Lieb} \cong \frac{\Lief/\Lier}{\Lies/\Lier} \cong \frac{\Lief}{\Lies}$. Now set $\Liem \cong  \frac{[\Lief, \Lief]_{\Lie} \cap \Lier}{[\Lief, \Lier]_{\Lie}}$ and $\Lieq \cong  \frac{[\Lief, \Lief]_{\Lie} \cap \Lies}{[\Lief, \Lier]_{\Lie}}$. Obviously $\Liem$ is a two-sided ideal of $\Lieq$.

Thus
\begin{equation} \label{eq 1}
\begin{array}{l}
{\normalsize [\Lieg, \Lieg]_{\Lie} \cap \Lieb} \cong [\frac{\Lief}{\Lier}, \frac{\Lief}{\Lier}]_{\Lie} \cap \frac{\Lies}{\Lier} \cong \frac{[\Lief, \Lief]_{\Lie}+ \Lier}{\Lier} \cap \frac{\Lies}{\Lier} \cong \frac{\left( [\Lief, \Lief]_{\Lie}+ \Lier \right) \cap \Lies}{\Lier} \cong \frac{\left( [\Lief, \Lief]_{\Lie} \cap \Lies \right) + \Lier}{\Lier} \cong \\
\\
\frac{ [\Lief, \Lief]_{\Lie} \cap \Lies}{\Lier \cap \left( [\Lief, \Lief]_{\Lie} \cap \Lies \right)} \cong \frac{ [\Lief, \Lief]_{\Lie} \cap \Lies}{[\Lief, \Lief]_{\Lie} \cap \Lier} \cong \frac{([\Lief, \Lief]_{\Lie} \cap \Lies)/{[\Lief, \Lier]_{\Lie}}}{([\Lief, \Lief]_{\Lie} \cap \Lier)/{[\Lief, \Lier]_{\Lie}}} \cong \frac{\Lieq}{\Liem}.
\end{array}
\end{equation}

Now second statement is obvious. For the third one,  since
\begin{equation} \label{eq 2}
{\cal M}^{\Lie}(\Liea) \cong \frac{\Lies \cap [\Lief, \Lief]_{\Lie}}{[\Lief, \Lies]_{\Lie}} \cong \frac{(\Lies \cap [\Lief, \Lief]_{\Lie})/{[\Lief, \Lier]_{\Lie}}}{[\Lief, \Lies]_{\Lie}/{[\Lief, \Lier]_{\Lie}}} \cong \frac{\Lieq}{{[\Lief, \Lies]_{\Lie}}/{[\Lief, \Lier]_{\Lie}}},
 \end{equation}
 then ${\cal M}^{\Lie}(\Liea)$ is the image of $\Lieq$ under some homomorphism whose kernel is  $\frac{[\Lief, \Lies]_{\Lie}}{[\Lief, \Lier]_{\Lie}}$.
\end{proof}

\begin{Co} \label{inequality}
Let $\Lieg$ be a finite-dimensional Leibniz algebra, $\Lieb$ be a two-sided ideal of $\Lieg$, and $\Liea \cong \Lieg/\Lieb$. Then
$${\rm dim}\left( {\cal M}^{\Lie}(\Liea) \right) \leq {\rm dim} \left( {\cal M}^{\Lie}(\Lieg) \right) + {\rm dim} \left( [\Lieg, \Lieg]_{\Lie} \cap \Lieb \right)$$
\end{Co}
\begin{proof}
From equation (\ref{eq 1}) we have the short exact sequence of vector spaces
$$0 \to \Liem \to \Lieq \to [\Lieg, \Lieg]_{\Lie} \cap \Lieb \to 0$$
hence ${\rm dim}(\Lieq) = {\rm dim}(\Liem) + {\rm dim}([\Lieg, \Lieg]_{\Lie} \cap \Lieb) = {\rm dim}({\cal M}^{\Lie}(\Lieg)) + {\rm dim}([\Lieg, \Lieg]_{\Lie} \cap \Lieb)$.

On the other hand, equation (\ref{eq 2}) implies that ${\rm dim}({\cal M}^{\Lie}(\Liea)) \leq {\rm dim}(\Lieq)$, which ends the proof.
\end{proof}


\begin{Th}
Let $\Lieg$ be a finite-dimensional Leibniz algebra and $\Lieb$ be a  \Lie-central two-sided ideal of $\Lieg$ (i.e. $\Lieb \subseteq Z_{\Lie}(\Lieg)$ such that $\Liea \cong \Lieg/\Lieb$. Then
$${\rm dim}\left( {\cal M}^{\Lie}(\Lieg) \right) +  {\rm dim}\left( [\Lieg, \Lieg]_{\Lie} \cap \Lieb \right) \leq
{\rm dim}\left( {\cal M}^{\Lie}(\Liea) \right) +  {\rm dim}\left( \Lieb \otimes \Lieg_{\Lie} \right)$$
\end{Th}
\begin{proof}
From the proof of Proposition 2 in \cite{CKh} there is the exact sequence
$$\Lieb \otimes \Lieg_{\Lie} \to {\cal M}^{\Lie}(\Lieg) \to {\cal M}^{\Lie}(\Liea) \to \Lieb \to \Lieg_{\Lie} \to \Liea_{\rm Lie} \to 0$$
and, having in mind diagram (\ref{free present diagr}), there is an epimorphism $\sigma : \Lieb \otimes \Lieg_{\Lie} \twoheadrightarrow \frac{[\Lief, \Lies]_{\Lie}}{[\Lief, \Lier]_{\Lie}}$.

From the proof of Corollary \ref{inequality} and by equation (\ref{eq 2}) we have:
\[
\begin{array}{l}
{\rm dim}\left( {\cal M}^{\Lie}(\Lieg) \right) + {\rm dim}\left( [\Lieg, \Lieg]_{\Lie} \cap \Lieb \right) = {\rm dim}\left( \Lieq \right) =\\
 {\rm dim}\left( {\cal M}^{\Lie}(\Liea) \right) + {\rm dim}\left( \frac{[\Lief, \Lies]_{\Lie}}{[\Lief, \Lier]_{\Lie}} \right) \leq  {\rm dim}\left( {\cal M}^{\Lie}(\Liea) \right) +  {\rm dim}\left( \Lieb \otimes \Lieg_{\Lie} \right)
 \end{array}
 \]
\end{proof}

\begin{Th} \label{sequence}
Let $0 \to \Lier \to \Lief \stackrel{\rho}\to \Lieg \to 0$ be a free presentation of a Leibniz algebra $\Lieg$. Let $\Lien$ be a two-sided ideal of $\Lieg$ and $\Lies$ be a two-sided ideal of $\Lief$  such that $\Lien \cong \frac{\Lies + \Lier}{\Lier}$. Then the following sequence is exact and natural
\begin{equation} \label{exact seq}
0 \to \frac{\Lier \cap [\Lief, \Lies]_{\Lie}}{[\Lief, \Lier]_{\Lie} \cap [\Lief, \Lies]_{\Lie}} \stackrel{\pi}\to {\cal M}^{\Lie}(\Lieg) \stackrel{\sigma}\to {\cal M}^{\Lie} \left(\frac{\Lieg}{\Lien} \right) \stackrel{\tau}\to \frac{\Lien \cap [\Lieg, \Lieg]_{\Lie}}{[\Lieg, \Lien]_{\Lie}} \to 0
\end{equation}
\end{Th}
\begin{proof}
From the following commutative diagram of free presentations
\[
\xymatrix{
 & 0 \ar[d] && 0 \ar[d] & 0 \ar[d] &  \\
0 \ar[r] & \bullet \ar[rr] \ar[dd] & & {\Lier} \ar@{>->}[ld] \ar[dd] \ar[r] & \bullet \ar[dd] \ar[r] & 0\\
& & \Lies+\Lier \ar@{>>}[ldd] \ar@{>->}[rd]& & & \\
0 \ar[r] & \Lies \ar@{>->}[ru] \ar[rr] \ar[d] & &\Lief \ar@{>>}[rd] \ar[r] \ar[d] & \bullet \ar[r] \ar[d] & 0\\
0 \ar[r] & \Lien \ar[rr]\ar[d] && \Lieg \ar[r] \ar[d] & \Lieg/\Lien \ar[r] \ar[d]& 0\\
 & 0 &  & 0&0 &
}
\]
we follow that ${\cal M}^{\Lie}(\Lieg) \cong  \frac{\Lier \cap [\Lief, \Lief]_{\Lie}}{[\Lief, \Lier]_{\Lie}}$,  ${\cal M}^{\Lie} \left( \frac{\Lieg}{\Lien} \right) \cong  \frac{(\Lies + \Lier) \cap [\Lief, \Lief]_{\Lie}}{[\Lief, \Lies + \Lier]_{\Lie}}$, since $\frac{\Lieg}{\Lien} \cong \frac{\Lief/\Lier}{(\Lies+\Lier)/\Lier} \cong \frac{\Lief}{\Lies + \Lier}$.

On the other hand, we can rewrite
$$\frac{\Lien \cap [\Lieg,\Lieg]_{\Lie}}{[\Lieg,\Lien]_{\Lie}} \cong \frac{\frac{\Lies + \Lier}{\Lier} \cap [\frac{\Lief}{\Lier},\frac{\Lief}{\Lier}]_{\Lie}}{[\frac{\Lief}{\Lier}, \frac{\Lies + \Lier}{\Lier}]_{\Lie}} \cong
  \frac{\frac{\Lies + \Lier}{\Lier} \cap \frac{[\Lief, \Lief]_{\Lie}}{\Lier}}{\frac{[\Lief, \Lies + \Lier]_{\Lie}}{\Lier}} \cong
\frac{\frac{\Lies + \Lier}{\Lier} \cap \frac{[\Lief, \Lief]_{\Lie}+ \Lier}{\Lier}}{\frac{[\Lief, \Lies + \Lier]_{\Lie}}{\Lier}}  \cong
\frac{\left( \Lies + \Lier \right) \cap \left( [\Lief, \Lief]_{\Lie} + \Lier \right)}{[\Lief, \Lies]_{\Lie} + \Lier}$$
Then it suffices to show the following sequence is exact:
$$0 \to \frac{\Lier \cap [\Lief, \Lies]_{\Lie}}{[\Lief, \Lier]_{\Lie} \cap [\Lief, \Lies]_{\Lie}} \stackrel{\pi} \to
\frac{\Lier \cap [\Lief, \Lief]_{\Lie}}{[\Lief, \Lier]_{\Lie}} \stackrel{\sigma} \to \frac{ \left( \Lies + \Lier \right) \cap [\Lief, \Lief]_{\Lie}}{[\Lief, \Lies + \Lier]_{\Lie}} \stackrel{\tau} \to \frac{ \left( \Lies + \Lier \right) \cap \left( [\Lief, \Lief]_{\Lie} + \Lier \right)}{[\Lief, \Lies]_{\Lie} + \Lier} \to 0$$

Define $\pi :  \frac{\Lier \cap [\Lief, \Lies]_{\Lie}}{[\Lief, \Lier]_{\Lie} \cap [\Lief, \Lies]_{\Lie}}  \to
\frac{\Lier \cap [\Lief, \Lief]_{\Lie}}{[\Lief, \Lier]_{\Lie}}$ by $\pi(x+ ([\Lief, \Lier]_{\Lie} \cap [\Lief, \Lies]_{\Lie})) = x + [\Lief, \Lier]_{\Lie}$. It is easy to check that $\pi$ is an injective well-defined linear map.
\medskip

Define $\sigma : \frac{\Lier \cap [\Lief, \Lief]_{\Lie}}{[\Lief, \Lier]_{\Lie}}  \to \frac{ \left( \Lies + \Lier \right) \cap [\Lief, \Lief]_{\Lie}}{[\Lief, \Lies + \Lier]_{\Lie}}$ by $\sigma(x + [\Lief, \Lier]_{\Lie}) = x + [\Lief, \Lies + \Lier]_{\Lie}$. Obviously
$\sigma$ is a well-defined linear map and $\sigma \circ \pi = 0$, consequently $\im( \pi) \subseteq \ker(\sigma)$.

On the other hand, given $x + [\Lief, \Lier]_{\Lie} \in \ker(\sigma)$, then $ x \in [\Lief, \Lies + \Lier]_{\Lie}$. Hence $x \in \Lier \cap [\Lief, \Lief]_{\Lie} \cap [\Lief, \Lies + \Lier]_{\Lie} = \Lier \cap [\Lief, \Lies + \Lier]_{\Lie}$. Thus $x + [\Lief, \Lier]_{\Lie} = [f,s+r]+[s+r,f] + [\Lief, \Lier]_{\Lie} \in  \frac{[\Lief, \Lies]_{\Lie}}{[\Lief, \Lier]_{\Lie}}$. Summarizing, $x + [\Lief, \Lier]_{\Lie} \in  \frac{\Lier \cap [\Lief, \Lies + \Lier]_{\Lie} \cap [\Lief, \Lies]_{\Lie}}{[\Lief, \Lier]_{\Lie}} = \frac{\Lier \cap  [\Lief, \Lies]_{\Lie}}{[\Lief, \Lier]_{\Lie}}$.

Then $x + \left( [\Lief, \Lier]_{\Lie} \cap [\Lief, \Lies]_{\Lie} \right) \in  \frac{\Lier \cap  [\Lief, \Lies]_{\Lie}}{[\Lief, \Lier]_{\Lie} \cap [\Lief, \Lies]_{\Lie}}$ satisfies that $\pi \left( x + \left( [\Lief, \Lier]_{\Lie} \cap [\Lief, \Lies]_{\Lie} \right) \right) = x + [\Lief, \Lier]_{\Lie}$, which implies that $\ker(\sigma) \subseteq \im( \pi)$.
\medskip

Define $\tau : \frac{ \left( \Lies + \Lier \right) \cap [\Lief, \Lief]_{\Lie}}{[\Lief, \Lies + \Lier]_{\Lie}} \to \frac{ \left( \Lies + \Lier \right) \cap \left( [\Lief, \Lief]_{\Lie} + \Lier \right)}{[\Lief, \Lies]_{\Lie} + \Lier}$ by $\tau \left(x+[\Lief, \Lies + \Lier]_{\Lie} \right) = x+([\Lief, \Lies]_{\Lie} + \Lier)$. $\tau$ is a well-defined linear map such that $\tau \circ \sigma =0$, then $\im(\sigma) \subseteq \ker(\tau)$.

For the converse, let $x + [\Lief, \Lies + \Lier]_{\Lie} \in \frac{ \left( \Lies + \Lier \right) \cap [\Lief, \Lief]_{\Lie}}{[\Lief, \Lies + \Lier]_{\Lie}}$ such that $\tau(x + [\Lief, \Lies + \Lier]_{\Lie}) = x + ([\Lief, \Lies]_{\Lie} + \Lier ) = \overline{0}$.

We need to prove that $x + [\Lief, \Lies + \Lier]_{\Lie} \in \im(\sigma)$. This occurs only if $x + [\Lief, \Lies + \Lier]_{\Lie}$ $ \in \frac{\Lier \cap [\Lief, \Lief]_{\Lie}}{[\Lief, \Lies + \Lier]_{\Lie}}$, so it  suffices to show that $x + [\Lief, \Lies + \Lier]_{\Lie} = r + [\Lief, \Lies + \Lier]_{\Lie}$ for some $r \in \Lier$.

Since $x \in [\Lief, \Lies]_{\Lie} + \Lier$, then $ x -r \in [\Lief, \Lies]_{\Lie}$ for some $r \in \Lier$. Thus $x-r + [\Lief, \Lies + \Lier]_{\Lie} = \overline{0}$, i.e. $x+ [\Lief, \Lies + \Lier]_{\Lie}  = r + [\Lief, \Lies + \Lier]_{\Lie}$. Consequently $x \in \frac{\Lier \cap  [\Lief, \Lief]_{\Lie}}{[\Lief, \Lies + \Lier]_{\Lie}}$.

Finally, $\tau$ is surjective. Namely, for $x + ([\Lief, \Lies]_{\Lie}  + \Lier) \in  \frac{ \left( \Lies + \Lier \right) \cap \left( [\Lief, \Lief]_{\Lie} + \Lier \right)}{[\Lief, \Lies]_{\Lie} + \Lier}$, we have that $x \in  \Lies + \Lier$ and $x \in  \frac{[\Lief, \Lief]_{\Lie} + \Lier}{[\Lief, \Lies]_{\Lie} + \Lier} \cong \frac{[\Lief, \Lief]_{\Lie}}{[\Lief, \Lies]_{\Lie} + \Lier}$. Hence $x \in (\Lies + \Lier) \cap [\Lief, \Lief]_{\Lie}$ and $\tau(x+ [\Lief, \Lies+ \Lier]_{\Lie}) = x  + ([\Lief, \Lies]_{\Lie} + \Lier)$.
\end{proof}

\begin{Co}
Let $\Lieg$ be a \Lie-nilpotent Leibniz algebra of class $k \geq 2$, then the following sequence is exact and natural:
$$0 \to \frac{\Lief^{[k+1]}}{[\Lief, \Lier]_{\Lie} \cap \Lief^{[k+1]}} \to {\cal M}^{\Lie} \left(\Lieg \right) \to {\cal M}^{\Lie} \left(\frac{\Lieg}{\Lieg^{[k]}} \right) \to \Lieg^{[k]} \to 0$$
\end{Co}
\begin{proof}
Take $\Lien = \Lieg^{[k]}$ and $\Lies = \Lief^{[k]}$ in Theorem \ref{sequence}. Then $\Lien = \Lieg^{[k]} \cong \frac{\Lief^{[k]}}{\Lier} \cong \frac{\Lief^{[k]}+\Lier}{\Lier} = \frac{\Lies+\Lier}{\Lier}$ and $[\Lief, \Lies]_{\Lie} = [\Lief, \Lief^{[k]}]_{\Lie} = \Lief^{[k+1]} \subseteq \Lier$. Now exact sequence (\ref{exact seq}) concludes the proof.
\end{proof}

\begin{Co}
Let $\Lien$ be a two-sided ideal of a  finite-dimensional Leibniz algebra $\Lieg$. Then
$${\rm dim} \left( {\cal M}^{\rm Lie} \left( \frac{\Lieg}{\Lien} \right) \right) +  {\rm dim} \left( \frac{\Lier \cap [\Lief, \Lies]_{\Lie}}{[\Lief, \Lier]_{\Lie} \cap [\Lief, \Lies]_{\Lie}}  \right) = {\rm dim} \left( {\cal M}^{\rm Lie} \left( \Lieg \right) \right) + {\rm dim} \left( \frac{\Lien \cap [\Lieg, \Lieg]_{\Lie}}{[\Lieg, \Lien]_{\Lie}}  \right)$$
\end{Co}
\begin{proof}
From exact sequence (\ref{exact seq}) we have ${\rm dim} \left( {\cal M}^{\rm Lie} \left( \frac{\Lieg}{\Lien} \right) \right) = {\rm dim} \left( \im(\sigma) \right) + {\rm dim} \left( \frac{\Lien \cap [\Lieg, \Lieg]_{\Lie}}{[\Lieg, \Lien]_{\Lie}}  \right) = {\rm dim} \left( {\cal M}^{\rm Lie} \left( \Lieg \right) \right) - {\rm dim} \left( \frac{\Lier \cap [\Lief, \Lies]_{\Lie}}{[\Lief, \Lier]_{\Lie} \cap [\Lief, \Lies]_{\Lie}}  \right) + {\rm dim} \left( \frac{\Lien \cap [\Lieg, \Lieg]_{\Lie}}{[\Lieg, \Lien]_{\Lie}}  \right)$
\end{proof}

\begin{De}
Let $\Lieq$ be a \Lie-nilpotent Leibniz algebra of class $k$. An extension of Leibniz algebras $0 \to \Lien \to \Lieg \stackrel{\pi}\to \Lieq \to 0$ is said to be of class $k$ if $\Lieg$ is nilpotent of class $k$.
\end{De}

\begin{Th}
A \Lie-central extension $0 \to \Lien \to \Lieg \stackrel{\pi}\to \Lieq \to 0$ is of class $k$ if and only if $\theta : {\cal M}^{\rm Lie} \left({\Lieg} \right) \to \Lien$  vanishes over $\ker(\tau)$, where $\tau : {\cal M}^{\rm Lie} \left( {\Lieq} \right)\to {\cal M}^{\rm Lie} \left( \Lieq/\Lieq^{[k]} \right)$ is induced by the canonical projection $\Lieq \twoheadrightarrow  \Lieq^{[k]}$.
\end{Th}
\begin{proof}
Consider the following diagrams of free presentations:
\[
\xymatrix{& & 0 \ar[d] & 0 \ar[ld]  &   &      & & 0 \ar[d] & 0 \ar[ld]  &     \\
&  & {\Lier} \ar[d]\ar[ld] & &              &  &  & {\Lies} \ar[d]\ar[ld] & &  \\
0 \ar[r] & {\Lies} \ \ar[r] \ar[d] & {\Lief} \ar[d]^{\rho} \ar[rd]^{\pi \circ \rho} & &    &  0 \ar[r] & {\Liet} \ \ar[r] \ar[d] & {\Lief} \ar[d]^{\pi \circ \rho} \ar[rd]^{pr \circ \pi \circ \rho} & &   \\
0 \ar[r]& {\Lien} \ar[r] \ar[d]& {\Lieg} \ar[r]^{\pi} \ar[d]& {\Lieq} \ar[r] \ar[dr] & 0   &  0 \ar[r]& {\Lieq}^{[k]} \ar[r] \ar[d]& {\Lieq} \ar[r]^{pr} \ar[d]& {\Lieq}/{\Lieq}^{[k]} \ar[r] \ar[dr] & 0 \\
 & 0 & 0 &  & 0 & & 0 & 0 &  & 0
}
\]
then $\theta : {\cal M}^{\rm Lie} \left( \Lieq \right) = \frac{\Lies \cap [\Lief, \Lief]_{\Lie}}{[\Lief, \Lies]_{\Lie}} \to \Lien$, given by $\theta(x+[{\Lief},{\Lies}]_{\Lie})= \rho(x)$, is well-defined and ${\ker}(\tau) \cong \frac{[{\Lief},\Liet]_{\Lie}}{[{\Lief},{\Lies}]_{\Lie}}$.

Assume that {\Lieg} is \Lie-nilpotent of class $k$ and consider $x+[{\Lief},{\Lies}]_{\Lie} \in {\ker}(\tau)$. Then $\theta(x+ [{\Lief},{\Lies}]_{\Lie}) = \rho(x) = 0$ since $\rho(x) \in [\rho({\Lief}), \rho({\Liet})]_{\Lie} \subseteq [{\Lieg}^{[k]}+{\Lien}, {\Lieg}]_{\Lie} = {\Lieg}^{[k+1]} = 0$. For the last inclusion is necessary to have in mind that $\pi \circ \rho({\Liet}) \subseteq {\Lieq}^{[k]} = \pi({\Lieg}^{[k]})$ and consequently $\rho({\Liet}) \subseteq {\Lieg}^{[k]} +{\Lien}$.

Conversely, ${\Lieg}^{[k+1]} = [{\Lieg}^{[k]}, {\Lieg}]_{\Lie} = [\rho({\Lief}^{[k]}), \rho({\Lief})]_{\Lie} \subseteq \rho[{\Liet},{\Lief}]_{\Lie}  =0$  since $[{\Liet},{\Lief}]_{\Lie} \subseteq {\Lier}$ because $\theta$ vanishes over ${\ker}(\tau)$. For the last inclusion is necessary to have in mind that $\pi \circ \rho({\Lief}^{[k]}) \subseteq {\Lieq}^{[k]}$, hence ${\Lief}^{[k]} \subseteq {\Liet}$.
\end{proof}

\begin{Pro}
Let $\Lieg$ be a $\Lie$-nilpotent Leibniz algebra and $f : \Lieg \twoheadrightarrow \Lieq$ be a surjective homomorphism of Leibniz algebras. If $\ker(f) \subseteq [\Lieg, \Lieg]_{\Lie}$ and ${\cal M}^{\Lie}(\Lieq) =0$, then $f$ is an isomorphism. In particular, if ${\cal M}^{\Lie}(\Lieg/[\Lieg,\Lieg]_{\Lie}) =0$, then ${\cal M}^{\Lie}(\Lieg) =0$.
\end{Pro}
\begin{proof}
Let $\Lien = \ker(f)$, then ${\cal M}^{\Lie}(\Lieg/\Lien) =0$.
From exact sequence (\ref{exact seq}) we have that $\Lien \cap [\Lieg, \Lieg]_{\Lie} \subseteq [\Lieg, \Lien]_{\Lie}$, then $\Lien \subseteq [\Lieg, \Lien]_{\Lie}$. Obviously $\supseteq$ is true, then  $\Lien = [\Lieg, \Lien]_{\Lie}$.

Let $\Lien^{[i]}=[\Lien^{[i-1]},\Lieg]_{\Lie}$ be the $i$-th term of the lower $\Lie$-central series of $\Lieg$ relative to $\Lien$ (see \cite[Definition 11]{CKh} for details). Obviously $\Lien = \Lien^{[i]} \subseteq \Lieg^{[i]}$, for all $i \in \mathbb{N}$.

Since $\Lieg$ is $\Lie$-nilpotent, there exists $k \in \mathbb{N}$ such that $\Lieg^{[k]}=0$, which implies that $\Lien=0$ and, consequently, $f$ is an isomorphism.
\end{proof}

\section{\Lie-stem covers} \label{stem cover}

The study of different types of \Lie-central extensions together with its corresponding characterizations is the subject of section 3.3 in \cite{CKh}. To summarize, a $\Lie$-central extension $f : \Lieg \twoheadrightarrow \Lieq$ is said to be a $\Lie$-{\it stem extension} if $\Lieg_{\Lie} \cong \Lieq_{\Lie}$. Additionally,  if  the induced map ${\cal M}^{\rm Lie} (\Lieg) \to {\cal M}^{\rm Lie} (\Lieq)$ is the zero map, then $f : \Lieg \twoheadrightarrow \Lieq$ is said to be a $\Lie$-{\it stem cover}. In this last case, $\Lieg$ is said to be a $\Lie$-{\it cover} or a $\Lie$-{\it covering} algebra.

A $\Lie$-stem extension $f : \Lieg \twoheadrightarrow \Lieq$ is characterized by the fact $\Lien \subseteq \Lieg^{\rm ann}$, equivalently, the map $\theta^{\ast}(\Lieg) : {\cal M}^{\rm Lie}(\Lieq) \to \Lien$ is an epimorphism. When $\theta$ is an isomorphism, then the $\Lie$-stem extension is a $\Lie$-stem cover (see \cite[Proposition 5, Proposition 6]{CKh} for details).

Now we are going to analyze the interplay between $\Lie$-stem covers and the Schur $\Lie$-multiplier.

\begin{Le} \label{lema 1}
Let $0 \to \Lier \to \Lief \stackrel{\rho} \to \Lieg \to 0$ be a free presentation of a Leibniz algebra $\Lieg$ and let $0 \to \Liem \to \Liep \stackrel{\theta} \to \Lieq \to 0$ be a $\Lie$-central extension of another Leibniz algebra $\Lieq$. Then for each homomorphism $\alpha : \Lieg \to \Lieq$, there exists a homomorphism $\beta : \frac{\Lief}{[\Lief, \Lier]_{\Lie}} \to \Liep$ such that $\beta \left( \frac{\Lier}{[\Lief, \Lier]_{\Lie}} \right) \subseteq \Liem$ and the following diagram is commutative:
\[
\xymatrix{
0 \ar[r] & \frac{\Lier}{[\Lief, \Lier]_{\Lie}} \ar[r] \ar[d]^{\beta_{\mid}} & \frac{\Lief}{[\Lief, \Lier]_{\Lie}} \ar[r]^{\overline{\rho}} \ar[d]^{\beta} & \Lieg \ar[r] \ar[d]^{\alpha} & 0\\
0 \ar[r] & \Liem \ar[r] & \Liep \ar[r]^{\psi} & \Lieq \ar[r] & 0
}
\]
\end{Le}
where $\overline{\rho}$ is the natural epimorphism induced by $\rho$.

\begin{proof}
Since $\Lief$ is a free Leibniz algebra, then there exists $\omega : \Lief \to \Liep$ such that $\psi \circ \omega = \alpha \circ \rho$.

On the other hand, $\psi ( \omega(\Lier)) = \alpha(\rho(\Lier))=0$, hence $\omega(\Lier) \subseteq \Liem$, which implies the vanishing of $\omega$ over $[\Lief, \Lier]_{\Lie}$. So $\omega$ induces $\beta : \frac{\Lief}{[\Lief, \Lier]_{\Lie}} \to \Liep$ and, for any $r \in \Lier$,  $\beta(r+[\Lief,\Lier]_{\Lie}) = \omega(r) \in \Liem$.
\end{proof}

\begin{Th} \label{Lie stem cover}
Let $\Lieg$ be a Leibniz algebra such that ${\cal M}^{\Lie}(\Lieg)$ is finite-dimensional and let $0 \to \Lier \to \Lief \stackrel{\rho} \to \Lieg \to 0$ be a free presentation of $\Lieg$. Then the extension  $0 \to \Liem \to \Liep \stackrel{\psi} \to \Lieg \to 0$ is a $\Lie$-stem cover if and only if there exists a two-sided ideal $\Lies$ of $\Lief$ such that
\begin{enumerate}
\item[(a)] $\Liep \cong \frac{\Lief}{\Lies}$ and $\Liem \cong \frac{\Lier}{\Lies}$.
\item[(b)] $\frac{\Lier}{[\Lief, \Lier]_{\Lie}} \cong {\cal M}^{\Lie}(\Lieg) \oplus \frac{\Lies}{[\Lief, \Lier]_{\Lie}}$.
\end{enumerate}
\end{Th}

\begin{proof}
Let $0 \to \Liem \to \Liep \stackrel{\psi} \to \Lieg \to 0$ be a $\Lie$-stem cover. By Lemma \ref{lema 1}, there exists a homomorphism $\beta : \frac{\Lief}{[\Lief, \Lier]_{\Lie}} \to \Liep$ such that $\psi \circ \beta = \overline{\rho}$ and $\beta(\frac{\Lier}{[\Lief,\Lier]_{\Lie}}) \subseteq \Liem$.

Since $\Liep = \im(\beta) + \Liem$ and $\Liem \subseteq Z_{\Lie}(\Liep)$, then by \cite[Proposition 5 (e)]{CKh} $\Liem \subseteq \Liep^{\rm ann} = [\im(\beta) + \Liem, \im(\beta) + \Liem] \subseteq \im(\beta)$. Thus $\beta$ is surjective and  $\beta(\frac{\Lier}{[\Lief,\Lier]_{\Lie}}) = \Liem$.

Set $\ker(\beta) = \frac{\Lies}{[\Lief,\Lier]_{\Lie}}$, then $\Liep \cong \frac{ {\Lief}/{[\Lief,\Lier]_{\Lie}}}{{\Lies}/{[\Lief,\Lier]_{\Lie}} }\cong \frac{\Lief}{\Lies}$ and $\Liem \cong \frac{ {\Lier} / [\Lief,\Lier]_{\Lie}}{{\Lies}/{[\Lief,\Lier]_{\Lie}}} \cong \frac{\Lier}{\Lies}$.

Now it remains to show statement {\it (b)}. Clearly $\beta \left( {\cal M}^{\Lie}(\Lieg) \right) = \beta \left( \frac{\Lier \cap [\Lief, \Lief]_{\Lie}}{[\Lief, \Lier]_{\Lie}}  \right) \subseteq \beta \left( \frac{\Lier} {[\Lief, \Lier]_{\Lie}} \right) \cap \beta \left( \frac{ [\Lief, \Lief]_{\Lie}}{[\Lief, \Lier]_{\Lie}}  \right) = \Liem \cap [\Liep, \Liep]_{\Lie} = \Liem$.

Conversely  $\Liem \subseteq \beta \left( {\cal M}^{\Lie}(\Lieg) \right)$. Indeed, in one side  $\Liem = \beta \left( \frac{\Lier} {[\Lief, \Lier]_{\Lie}} \right)$ and, in the other side, $\Liem \subseteq \Liep^{\rm ann} \subseteq [\Liep, \Liep]_{\Lie}$, therefore for any $m \in \Liem$, there exists $x \in [\Lief, \Lief]_{\Lie}$ such that $\beta(x+[\Lief, \Lier]_{\Lie}) = m$. Then $\beta(x+[\Lief, \Lier]_{\Lie}) = m = \beta(r+[\Lief, \Lier]_{\Lie})$, thus $x-r+[\Lief, \Lier]_{\Lie} \in \ker(\beta) =  \frac{\Lies}{[\Lief,\Lier]_{\Lie}} \subseteq \frac{\Lier}{[\Lief,\Lier]_{\Lie}}$ which implies that $x+[\Lief, \Lier]_{\Lie} \in  \frac{\Lier}{[\Lief,\Lier]_{\Lie}}$, hence $x \in \Lier$.
Summarizing, $m = \beta(x+[\Lief, \Lier]_{\Lie})$, whit $x \in \Lier \cap [\Lief, \Lief]_{\Lie}$, i.e. $m \in \beta \left( {\cal M}^{\Lie}(\Lieg) \right)$.

Therefore, $\beta$ restricts to an epimorphism from ${\cal M}^{\Lie}(\Lieg)$ onto $\Liem$ and we have the following commutative diagram:
\[
\xymatrix{
{\cal M}^{\Lie}(\Lieg) \cap \frac{\Lies}{[\Lief, \Lier]_{\Lie}} \ar@{>->}[rd] & & \frac{\Lies}{[\Lief, \Lier]_{\Lie}} \ar@{>->}[d] \ar@{=}[r]  &  \frac{\Lies}{[\Lief, \Lier]_{\Lie}} \ar@{>->}[d] & \\
& {\cal M}^{\Lie}(\Lieg) \ar@{>>}[rd]^{\beta_{\mid}}  \ar@{^{(}->}[r] & \frac{\Lier}{[\Lief, \Lier]_{\Lie}} \ar@{>->}[r] \ar@{>>}[d]^{\beta_{\mid}} & \frac{\Lief}{[\Lief, \Lier]_{\Lie}} \ar@{>>}[r]^{\overline{\rho}} \ar[d]^{\beta} & \Lieg  \ar@{=}[d]\\
&  & \Liem \ar@{>->}[r] & \Liep \ar@{>>}[r]^{\psi} & \Lieq
}
\]
Now, for any $r + [\Lief, \Lier]_{\Lie} \in \frac{\Lier}{[\Lief, \Lier]_{\Lie}}$, since $\beta_{\mid} \left( r + [\Lief, \Lier]_{\Lie} \right) \in \Liem$, then there exists $x + [\Lief, \Lier]_{\Lie} \in {\cal M}^{\Lie}(\Lieg)$ such that $\beta_{\mid} \left( x + [\Lief, \Lier]_{\Lie} \right) = \beta_{\mid} \left( r + [\Lief, \Lier]_{\Lie} \right)$, hence $r- x + [\Lief, \Lier]_{\Lie} \in \ker(\beta_{\mid}) = \frac{\Lies}{[\Lief, \Lier]_{\Lie}}$, consequently $r + [\Lief, \Lier]_{\Lie} = x + [\Lief, \Lier]_{\Lie} + s + [\Lief, \Lier]_{\Lie}$, i.e.
$$\frac{\Lier}{[\Lief, \Lier]_{\Lie}} \cong {\cal M}^{\Lie}(\Lieg) + \frac{\Lies}{[\Lief, \Lier]_{\Lie}}$$
Moreover, this sum is a direct sum, because ${\rm dim} \left( {\cal M}^{\Lie}(\Lieg) \cap \frac{\Lies}{[\Lief, \Lier]_{\Lie}} \right) + {\rm dim} \left( \Liem \right) = {\rm dim} \left( {\cal M}^{\Lie}(\Lieg) \right) = {\rm dim} \left( \Liem \right)$, which implies that ${\cal M}^{\Lie}(\Lieg) \cap \frac{\Lies}{[\Lief, \Lier]_{\Lie}} = 0$.

Conversely, assume the existence of a two-sided ideal $\Lies$ of $\Lief$ which satisfies statements {\it (a)} and {\it (b)}. Consider $\Liem = \frac{\Lier}{\Lies}, \Liep = \frac{\Lief}{\Lies}$, then $\Lieg \cong \frac{\Liep}{\Liem} \cong \frac{\Lief /\Lies}{\Lier/\Lies}$, and  $0 \to \Liem \to \Liep  \to \Lieg \to 0$ obviously is a $\Lie$-central extension. From {\it (b)} we have the split short exact sequence $0 \to \frac{\Lies}{[\Lief, \Lier]_{\Lie}} \to  \frac{\Lier}{[\Lief, \Lier]_{\Lie}} \to  {\cal M}^{\Lie}(\Lieg)  \to 0$, hence ${\cal M}^{\Lie}(\Lieg) \cong \frac{ {\Lier}/{[\Lief, \Lier]_{\Lie}}}{{\Lies}/{[\Lief, \Lier]_{\Lie}}} \cong \frac{\Lier}{\Lies} \cong \Liem$. Proposition 6 in \cite{CKh} completes the proof.
\end{proof}
\bigskip

Previously to the following result, we need recall some notions concerning $\Lie$-isoclinism of Leibniz algebras from \cite{BC}.

Consider the $\Lie$-central extensions $(g_i) : 0 \to \Lien_i \stackrel{\chi_i}\to \Lieg_i \stackrel{\pi_i} \to \Lieq_i \to 0, i=1, 2,$
Let be $C_i : \Lieq_i \times \Lieq_i \to [\Lieg_i, \Lieg_i]_{\Lie}$ given by $C_i(q_{i1},q_{i2})=[g_{i1},g_{i2}]+[g_{i2},g_{i1}]$, where $\pi_i(g_{ij})=q_{ij}, i, j= 1, 2$, the $\Lie$-commutator map associated to the extension $(g_i)$.

\begin{De} \label{isoclinic}
The $\Lie$-central extensions $(g_1)$ and $(g_2)$ are said to be \Lie-isoclinic when there exist isomorphisms $\eta : \Lieq_1 \to \Lieq_2$ and $\xi : [\Lieg_1, \Lieg_1]_{\Lie} \to [\Lieg_2, \Lieg_2]_{\Lie}$ such that the following diagram is commutative:
\begin{equation}  \label{square isoclinic}
\xymatrix{
\Lieq_1 \times \Lieq_1 \ar[r]^{C_1} \ar[d]_{\eta \times \eta} & [\Lieg_1, \Lieg_1]_{\Lie} \ar[d]^{\xi}\\
\Lieq_2 \times \Lieq_2 \ar[r]^{C_2} & [\Lieg_2, \Lieg_2]_{\Lie}
}
\end{equation}

The pair $(\eta, \xi)$ is called a \Lie-isoclinism from $(g_1)$ to $(g_2)$ and will be denoted by $(\eta, \xi) : (g_1) \to (g_2)$.
\end{De}

\begin{Co}
Let $\Lieg$ be a Leibniz algebra such that its Schur $\Lie$-multiplier is finite-dimensional. Then all $\Lie$-stem covers of $\Lieg$ are \Lie-isoclinic.
\end{Co}
\begin{proof}
Let $0 \to \Lier \to \Lief \stackrel{\rho} \to \Lieg \to 0$ be a free presentation of $\Lieg$. Let $0 \to \Liem \to \Liep \stackrel{\psi} \to \Lieg \to 0$ is a $\Lie$-stem cover. By Theorem \ref{Lie stem cover} there exists an epimorphism $\beta : \frac{\Lief}{[\Lief, \Lier]_{\Lie}} \to \Liep$ and a two-sided ideal $\Lies$ of $\Lief$ such that $\frac{\Lier}{[\Lief, \Lier]_{\Lie}} \cong {\cal M}^{\Lie}(\Lieg) \oplus \ker(\beta)$ and $\ker(\beta) = \frac{\Lies}{[\Lief, \Lier]_{\Lie}}$. Moreover $\ker(\beta) \cap \left[ \frac{\Lief}{[\Lief, \Lier]_{\Lie}}, \frac{\Lief}{[\Lief, \Lier]_{\Lie}} \right]_{\Lie} = \frac{\Lies}{[\Lief, \Lier]_{\Lie}} \cap \frac{[\Lief, \Lief]_{\Lie}}{[\Lief, \Lier]_{\Lie}} = \frac{\Lies \cap [\Lief, \Lief]_{\Lie}}{[\Lief, \Lier]_{\Lie}} = 0$. Now Propositions 3.20 (b) and 3.5 in \cite{BC} complete the proof.
\end{proof}

\begin{Co} \label{one stem cover}
Any finite-dimensional Leibniz algebra has at least one $\Lie$-cover.
\end{Co}
\begin{proof}
Let $0 \to \Lier \to \Lief \stackrel{\rho} \to \Lieg \to 0$ be a free presentation of $\Lieg$ and $\frac{\Lies}{[\Lief, \Lier]_{\Lie}}$ be a complement of ${\cal M}^{\Lie}(\Lieg)$ in $\frac{\Lier}{[\Lief, \Lier]_{\Lie}}$, for a suitable two-sided ideal $\Lies$ of $\Lief$. Then $\frac{\Lief}{\Lies}$ is a $\Lie$-cover of $\Lieg$ by Theorem \ref{Lie stem cover}.
\end{proof}


\begin{Le} \label{epimorphism}
Let $\Lieg$ be a Leibniz algebra and
\[
\xymatrix{
0 \ar[r] & \Liem_1 \ar[r] \ar[d]^{\alpha} & \Liep_1 \ar[r] \ar[d]^{\beta} & \Lieg \ar[r] \ar[d]^{\gamma} & 0\\
0 \ar[r] & \Liem_2 \ar[r] & \Liep_2 \ar[r] & \Lieg \ar[r] & 0
}
\]
\end{Le}
be a commutative diagram of short exact sequences of Leibniz algebras such that the bottom row is a $\Lie$-stem extension. If the homomorphism $\gamma$ is surjective, then $\beta$ is a surjective homomorphism as well.

\begin{proof}
Obviously $\Liep_2 = \im(\beta) + \Liem_2$. Hence $[\Liep_2, \Liep_2]_{\Lie} = [\im(\beta),\im(\beta)]_{\Lie}$.
By \cite[Proposition 5 (e)]{CKh}, $\Liem_2 \subseteq \Liep_2^{\rm ann} \subseteq [\Liep_2, \Liep_2]_{\Lie}  = [\im(\beta),\im(\beta)]_{\Lie}$.
Therefore  $\Liep_2 \subseteq \im(\beta) +  [\im(\beta),\im(\beta)]_{\Lie}$, i.e. $\beta$ is surjective.
\end{proof}

\begin{Le}
Let $0 \to \Lier \to \Lief \stackrel{\rho} \to \Lieg \to 0$ be a free presentation of a Leibniz algebra $\Lieg$. Then every $\Lie$-stem extension of $\Lieg$ is epimorphic image of $\frac{\Lief}{[\Lief, \Lier]_{\Lie}}$.
\end{Le}

\begin{proof}
Given a $\Lie$-stem extension $0 \to \Liem \to \Liep  \to \Lieg \to 0$, then Lemma \ref{lema 1} provides the following commutative diagram:
\[
\xymatrix{
0 \ar[r] & \frac{\Lier}{[\Lief, \Lier]_{\Lie}} \ar[r] \ar[d]^{\beta_{\mid}} & \frac{\Lief}{[\Lief, \Lier]_{\Lie}} \ar[r]^{\overline{\rho}} \ar[d]^{\beta} & \Lieg \ar[r] \ar@{=}[d] & 0\\
0 \ar[r] & \Liem \ar[r] & \Liep \ar[r] & \Lieq \ar[r] & 0
}
\]
Lemma \ref{epimorphism} implies that $\beta$ is surjective.
\end{proof}

\begin{Th} \label{hopfian}
Let $\Lieg$ be a Leibniz algebra such that ${\cal M}^{\Lie}(\Lieg)$ is finite-dimensional and let $0 \to \Liem_i \to \Liep_i \stackrel{\psi_i}\to  \Lieg \to 0, i=1, 2$, be two $\Lie$-stem covers of $\Lieg$. If $\eta : \Liep_1 \to \Liep_2$ is an epimorphism such that $\eta(\Liem_1) \subseteq \Liem_2$, the $\eta$ is an isomorphism.
\end{Th}

\begin{proof}
Let $0 \to \Lier \to \Lief \stackrel{\rho} \to \Lieg \to 0$ be a free presentation of $\Lieg$. By Theorem \ref{Lie stem cover} there exist two-sided ideals $\Lies_i, i = 1,2$, of $\Lief$ such that $\Liep_i \cong \frac{\Lief}{\Lies_i}; \Liem_i \cong \frac{\Lier}{\Lies_i}$ and $\frac{\Lier}{[\Lief, \Lier]_{\Lie}} \cong {\cal M}^{\Lie}(\Lieg) \oplus \frac{\Lies_i}{[\Lief, \Lier]_{\Lie}}, i = 1, 2$.

By Lemmas \ref{lema 1} and \ref{epimorphism} and the proof of Theorem \ref{Lie stem cover}, there exists an epimorphism $\theta :  \frac{\Lief}{[\Lief, \Lier]_{\Lie}} \to \Liep_2 \cong \frac{\Lief}{\Lies_2}$ such that $\ker(\theta) =  \frac{\Lies_2}{[\Lief, \Lier]_{\Lie}}$.

Since $\Lief$ is a free Leibniz algebra, then there exists a homomorphism $\overline{\delta} : \Lief \to \Liep_1$ such that $\psi_1 \circ \overline{\delta} = \pi$. Moreover $\overline{\delta}(\Lier) \subseteq \Liem_1$ and $\overline{\delta}$ vanishes on $[\Lief, \Lier]_{\Lie}$, consequently it induces a homomorphism $\delta' : \frac{\Lief}{[\Lief, \Lier]_{\Lie}} \to \Liep_1 \cong \frac{\Lief}{\Lies_1}$ such that $\delta' \circ pr = \overline{\delta}$, where $pr : \Lief \to  \frac{\Lief}{[\Lief, \Lier]_{\Lie}}$ is the canonical projection. Since $\psi_1 \circ \delta' = \overline{\pi}$, then Lemma \ref{epimorphism} implies that $\delta'$ is an epimorphism. Let $\ker(\delta') = \frac{\Liet}{[\Lief, \Lier]_{\Lie}}$ for some two-sided ideal $\Liet$ of $\Lier$.

Since $\theta \left(\frac{\Liet}{[\Lief, \Lier]_{\Lie}} \right) =  \eta \left(\delta' \left( \frac{\Liet}{[\Lief, \Lier]_{\Lie}} \right) \right) = 0$, then $\frac{\Liet}{[\Lief, \Lier]_{\Lie}} \subseteq \ker(\theta) = \frac{\Lies_2}{[\Lief, \Lier]_{\Lie}}$, therefore $\Liet \subseteq \Lies_2$.

From the following diagram
\[
\xymatrix{
  & \frac{\Liet}{[\Lief, \Lier]_{\Lie}} \ar@{>->}[d]^{j} \ar[ld]^{\sigma} & \\
\frac{\Lies_1}{[\Lief, \Lier]_{\Lie}}  \ar@{>->}[r]^i \ar@<1ex>[ur]^{\tau}& \frac{\Lier}{[\Lief, \Lier]_{\Lie}}  \ar@{>>}[r]^{\beta_{\mid}} \ar@{>>}[d]^{\delta'}& \Liem_1\\
 & \Liem_1&
}
\]
it follows that $\frac{\Lies_1}{[\Lief, \Lier]_{\Lie}} \cong \frac{\Liet}{[\Lief, \Lier]_{\Lie}}$ and, by Theorem \ref{Lie stem cover}, we have
${\cal M}^{\Lie}(\Lieg) \oplus \frac{\Lies_2}{[\Lief, \Lier]_{\Lie}} \cong  \frac{\Lier}{[\Lief, \Lier]_{\Lie}} \cong {\cal M}^{\Lie}(\Lieg) \oplus \frac{\Liet}{[\Lief, \Lier]_{\Lie}}$, which implies that $\Lies_2 \cong \Liet$. Since $\ker(\eta) \cong \frac{\Lies_2}{\Liet}$, then $\eta$ is an isomorphism.
\end{proof}

\begin{Co}
Every $\Lie$-stem cover of a Leibniz algebra $\Lieg$ with trivial $\Lie$-commu\-tator and finite-dimensional Schur $\Lie$-multiplier is Hopfian, that is, every epimorphism is an isomorphism.
\end{Co}
\begin{proof}
Let $0 \to \Lien \to \Lieg^{\ast} \to \Lieg \to 0$ be a $\Lie$-stem cover. Then there exists a two-sided ideal $\Liem$ of $\Lieg^{\ast}$ such that $\Liem = [\Lieg^{\ast}, \Lieg^{\ast}]_{\Lie}$ and $\Lieg \cong {\Lieg^{\ast}}/{\Liem}$.

Now, if $\eta : \Lieg^{\ast} \to \Lieg^{\ast}$ is an epimorphism, then $\eta(\Lien) = \Liem$. By Theorem \ref{hopfian}, $\eta$ is an isomorphism.
\end{proof}

\begin{Pro}
Let $0 \to \Liem_i \to \Liep_i \stackrel{\psi_i}\to  \Lieg \to 0, i=1, 2$, be two $\Lie$-stem covers of a finite-dimensional Leibniz algebra with finite-dimensional Schur $\Lie$-multiplier $\Lieg$. Then $Z_{\Lie}(\Liep_1)/\Liem_1 \cong Z_{\Lie}(\Liep_2)/\Liem_2$.
\end{Pro}
\begin{proof}
Let $0 \to \Lier \to \Lief \stackrel{\rho} \to \Lieg \to 0$ be a free presentation of $\Lieg$. By Corollary \ref{one stem cover} there exists a $\Lie$-cover $\Lieg^{\ast}$ of $\Lieg$, i.e. there is an exact sequence  $0 \to \Liem \to \Lieg^{\ast} \stackrel{\psi}\to  \Lieg \to 0$ such that $\Liem \subseteq Z_{\Lie}(\Lieg^{\ast}) \cap {\Lieg^{\ast}}^{\rm ann}$ and $ \Liem \cong {\cal M}^{\Lie}(\Lieg)$ (see \cite[Propositions 5 and 6]{CKh}).

By Theorem \ref{Lie stem cover}, there exists a two-sided ideal $\Lies$ such that $\Lieg^{\ast} \cong \frac{\Lief}{\Lies}$, $\Liem \cong \frac{\Lier}{\Lies}$ and $\frac{\Lief}{[\Lief, \Lier]_{\Lie}} \cong {\cal M}^{\Lie}(\Lieg) \oplus \frac{\Lies}{[\Lief, \Lier]_{\Lie}}$. Put $Z_{\Lie}(\frac{\Lief}{[\Lief, \Lier]_{\Lie}}) = \frac{\Liet}{[\Lief, \Lier]_{\Lie}}$, then $[\Lief, \Liet]_{\Lie} \subseteq [\Lief, \Lier]_{\Lie}$, thus $\frac{\Liet}{\Lies} \subseteq Z_{\Lie}(\frac{\Lief}{\Lies})$.

Conversely, for $x + \Lies \in Z_{\Lie}(\frac{\Lief}{\Lies})$, we must show that $x + \Lies \in \frac{\Liet}{\Lies}$.

 Indeed, for any $f + \Lies \in \frac{\Lief}{\Lies}$, $[x + \Lies, f+\Lies]+[f+\Lies, x+\Lies] = 0$, hence $[x,f]+[f,x] \in \Lies \cap [\Lief,\Lief]_{\Lie}$, for any $f \in \Lief$.

 To show that $x \in \Liet$ it is enough to prove that $x +[\Lief,\Lier]_{\Lie} \in Z_{\Lie}(\frac{\Lief}{[\Lief, \Lier]_{\Lie}}) = \frac{\Liet}{[\Lief, \Lier]_{\Lie}}$. But this fact holds since for any $f \in \Lief, [x,f]+[f,x] + [\Lief, \Lier]_{\Lie} = \overline{0}$, because $[x,f]+[f,x] \in \Lies \cap [\Lief, \Lief]_{\Lie}$ and by Theorem \ref{Lie stem cover} $\frac{\Lief}{[\Lief, \Lier]_{\Lie}} \cong \frac{\Lier \cap [\Lief, \Lief]_{\Lie}}{[\Lief, \Lier]_{\Lie}} \oplus \frac{\Lies}{[\Lief, \Lier]_{\Lie}}$, hence $\Lier \cap [\Lief, \Lief]_{\Lie} \cap \Lies \subseteq [\Lief, \Lier]_{\Lie}$, but $\Lies \subseteq \Lier$, then $\Lies \cap [\Lief, \Lief]_{\Lie} \subseteq [\Lief, \Lier]_{\Lie}$.

 Consequently, $\frac{\Liet}{\Lies} \cong Z_{\Lie}\left( \frac{\Lief}{\Lies} \right)$. From here $\frac{Z_{\Lie}\left( \Lieg^{\ast} \right)}{\Liem} \cong \frac{Z_{\Lie}\left( \Lief/ \Lies \right)}{\Lier/\Lies} \cong \frac{\Liet / \Lies}{\Lier / \Lies} \cong \frac{\Liet}{\Lier}$.

Applying this result to each $\Lie$-stem cover, we have $\frac{Z_{\Lie}(\Liep_1)}{\Liem_1} \cong \frac{\Liet}{\Lier} \cong  \frac{Z_{\Lie}(\Liep_2)}{\Liem_2}$.
 \end{proof}


 \section{The Schur $\Lie$-multiplier and the precise $\Lie$-center} \label{capability}
 The precise $\Lie$-center  was introduced in \cite{CKh} in order to characterize $\Lie$-capability of Leibniz algebras. Our aim in this section is to analyze its connections with the Schur $\Lie$-multiplier.

\begin{De} \cite[Definition 4]{CKh}
The precise $\Lie$-center $Z_{\Lie}^{\ast}(\Lieq)$ of a Leibniz algebra $\Lieq$ is the intersection of all two-sided ideals $f(Z_{\Lie}(\Lieg))$, where $f : \Lieg \twoheadrightarrow \Lieq$ is a $\Lie$-central extension.
\end{De}

\begin{Th}
Let $\Lieg$ be a Leibniz algebra with finite-dimensional Schur $\Lie$-multiplier and let $0 \to \Liem \to \Lieg^{\ast} \stackrel{\psi}\to \Lieg \to 0$ be a $\Lie$-stem cover. Then $Z_{\Lie}^{\ast}(\Lieg) = \psi(Z_{\Lie}(\Lieg^{\ast}))$.
\end{Th}
\begin{proof}
Let $0 \to \Lier \to \Lief \stackrel{\rho} \to \Lieg \to 0$ be a free presentation of $\Lieg$. By Theorem \ref{Lie stem cover}  there exists a two-sided ideal $\Lies$ of $\Lier$ such that $\Lieg^{\ast} \cong \frac{\Lief}{\Lies}, \Liem \cong \frac{\Lier}{\Lies}$ and  $\frac{\Lier}{[\Lief, \Lier]_{\Lie}} \cong {\cal M}^{\Lie}(\Lieg) \oplus \frac{\Lies}{[\Lief, \Lier]_{\Lie}}$.

Put $Z_{\Lie}(\Lieg^{\ast}) = Z_{\Lie}(\frac{\Lief}{\Lies}) \cong \frac{\Liet}{\Lies}$ for some two-sided ideal $\Liet$ of $\Lief$, then $[\Lief, \Liet]_{\Lie} \subseteq \Lies \cap [\Lief, \Lief]_{\Lie} = [\Lief, \Lier]_{\Lie}$, and hence $\frac{\Liet}{[\Lief, \Lier]_{\Lie}} \subseteq Z_{\Lie} \left( \frac{\Lief}{[\Lief, \Lier]_{\Lie}} \right)$.

On the other hand, if $z + [\Lief, \Lier]_{\Lie} \in  Z_{\Lie} \left( \frac{\Lief}{[\Lief, \Lier]_{\Lie}} \right)$, then for any $f \in \Lief$ we have that $[z,f]+[f,z] \in [\Lief, \Lier]_{\Lie} \subseteq \Lies$, so $z + \Lies \in Z_{\Lie} \left( \frac{\Lief}{\Lies} \right) = \frac{\Liet}{\Lies}$. Consequently, $Z_{\Lie} \left( \frac{\Lief}{[\Lief, \Lier]_{\Lie}} \right) \subseteq \frac{\Liet}{[\Lief, \Lier]_{\Lie}}$. This fact, together with the above one, actually provides an equality. Thus, thanks to \cite[Lemma 2]{CKh},  we have
$$Z_{\Lie}(\Lieg^{\ast}) = \overline{\rho} \left( Z_{\Lie} \left( \frac{\Lief}{[\Lief, \Lier]_{\Lie}} \right) \right) = \overline{\rho} \left( \frac{\Liet}{[\Lief, \Lier]_{\Lie}}  \right) = \rho(\Liet) = \psi \left( \frac{\Liet}{\Lies} \right) = \psi \left( Z_{\Lie}(\Lieg^{\ast}) \right)$$
\end{proof}

\begin{Th} \label{equivalences}
Let $\Lien$ be a $\Lie$-central two-sided ideal of a Leibniz algebra $\Lieg$. The following statements are equivalent:
\begin{enumerate}
\item[(a)] $\Lien \cap [\Lieg, \Lieg]_{\Lie} \cong {\cal M}^{\Lie}(\frac{\Lieg}{\Lien}) / {\cal M}^{\Lie}(\Lieg)$.
\item[(b)] $\Lien \subseteq Z_{\Lie}^{\ast}(\Lieg)$.
\item[(c)] The natural map $\sigma: {\cal M}^{\Lie}(\Lieg) \to {\cal M}^{\Lie}(\frac{\Lieg}{\Lien})$ is a monomorphism.
\end{enumerate}
\end{Th}
\begin{proof}
Statements {\it (a)} and {\it (c)} are equivalent thanks to exact sequence (\ref{exact seq}).

 {\it (b)} $\Leftrightarrow$ {\it (c)} Consider a diagram of free presentations similar to diagram (\ref{free present diagr}), from which immediately follows that $[\Lief, \Lies]_{\Lie} \subseteq \Lier$. Since $\ker(\sigma) \cong \frac{[\Lief, \Lies]_{\Lie}}{[\Lief, \Lier]_{\Lie}}$, then to prove the equivalence between {\it (b)} and {\it (c)} is enough to show that $[\Lief, \Lies]_{\Lie} = [\Lief, \Lier]_{\Lie}$ is equivalent to $\Lien \subseteq Z_{\Lie}^{\ast}(\Lieg)$.

 Set $\overline{\Lief} = \frac{\Lief}{[\Lief, \Lier]_{\Lie}}, \overline{\Lier} = \frac{\Lier}{[\Lief, \Lier]_{\Lie}}$ and $\overline{\Lies} = \frac{\Lies}{[\Lief, \Lier]_{\Lie}}$, then $[\Lief, \Lies]_{\Lie} = [\Lief, \Lier]_{\Lie}$ is equivalent to $\overline{\Lies} \subseteq Z_{\Lie}(\overline{\Lief})$. But, by Lemma 2 in \cite{CKh}, $Z_{\Lie}^{\ast}(\Lieg) = \overline{\rho} \left( Z_{\Lie}(\overline{\Lief}) \right)$. In consequence, $\overline{\rho}(\overline{\Lies}) \subseteq Z_{\Lie}^{\ast}(\Lieg)$ if and only if $\overline{\Lies} \subseteq Z_{\Lie}(\overline{\Lief})$. Then the result follows since $\overline{\rho}(\overline{\Lies}) = \Lien$.
\end{proof}

\begin{Co}
$Z_{\Lie}^{\ast}(\Lieg) = 0$ (i.e. $\Lieg$ is $\Lie$-capable, see \cite[Corollary 2]{CKh}) if and only if the natural map $\sigma_x: {\cal M}^{\Lie}(\Lieg) \to {\cal M}^{\Lie}(\frac{\Lieg}{\langle x \rangle})$ has non-trivial kernel for all non-zero elements $x \in Z_{\Lie}(\Lieg)$.
\end{Co}
\begin{proof}
Assume that $\ker(\sigma_x)$ is trivial for any non-zero element $x \in Z_{\Lie}(\Lieg)$. By theorem \ref{equivalences} $\langle x \rangle \subseteq Z_{\Lie}^{\ast}(\Lieg)$, so $Z_{\Lie}^{\ast}(\Lieg) \neq 0$.

For every non-zero element $x \in Z_{\Lie}(\Lieg)$, we have $0 \neq \langle x \rangle \nsubseteqq Z_{\Lie}^{\ast}(\Lieg) = 0$, then $\sigma_x$ cannot be a monomorphism.
\end{proof}


\section*{\bf Acknowledgements}
Authors were supported by Ministerio de Economía y Competitividad (Spain), grant MTM2016-79661-P (AEI/FEDER, UE, support included).


\begin{center}

\end{center}

\end{document}